\documentclass[a4paper,12pt,reqno]{amsart}

 \usepackage[T2A]{fontenc}
 \usepackage[utf8]{inputenc}
 \usepackage[ukrainian,english]{babel}
 \usepackage{amssymb,upref}
\usepackage{amsthm}

\usepackage{graphicx}
\usepackage{mathrsfs}

\usepackage[text={130mm,200mm}]{geometry}

\theoremstyle{plain}
\newtheorem{theorem}{Theorem}
\newtheorem*{theorem*}{Theorem}

\newtheorem*{corollary*}{Corollary}

\newtheorem*{lemma*}{Lemma}

\newtheorem*{proposition*}{Proposition}

\newtheorem*{conjecture*}{Conjecture}
\theoremstyle{definition}
\newtheorem{definition}{Definition}
\newtheorem*{definition*}{Definition}
\theoremstyle{remark}
\newtheorem{remark}{Remark}
\newtheorem*{remark*}{Remark}

\begin{document}

\title[Numeral systems with non-zero redundancy]{Numeral systems with non-zero redundancy and their applications in the theory of locally complex functions}

\author{S. O. Vaskevych}
\address[S. O. Vaskevych]{Institute of Mathematics of NAS of Ukraine, Kyiv, Ukraine\\
ORCID 0009-0005-3979-4543}
\email{svetaklymchuk@imath.kiev.ua}

\author{Yu.Yu. Vovk}
\address[Yu.Yu. Vovk]{K. D. Ushynskyi Chernihiv Regional Institute of Postgraduate Pedagogical Education, Chernihiv, Ukraine}
\email{freeeidea@ukr.net}

\author{O. M. Pratsiovytyi}
\address[O. M. Pratsiovytyi]{Ukrainian State Dragomanov University, Kyiv, Ukraine}
\email{o.m.pratsovytyi@udu.edu.ua}


% Key words

\thanks{Bukovinian Math. Journal. 2025, 13, 2, 152–160}
\thanks{This work was supported by a grant from the Simons Foundation (SFI-PD-Ukraine-00014586 V.S.)}

\begin{abstract}
In this paper we study representations of real numbers in a numeral system with the base $a>1$ and alphabet (digits set) $A\equiv\{0,1,...,r\}$, $a-1<r\in N$ given by
\[x=\sum\limits_{n=1}^{\infty}\frac{\alpha_n}{a^n}\equiv
\Delta^{r_a}_{\alpha_1\alpha_2...\alpha_n...}, \alpha_n\in A.\]
Since the alphabet is redundant  the numbers from the interval $[0;\frac{r}{a-1}]$ have not a single representation and can even have a continuous set of different representations.

We describe the geometry (topological and metric properties) of such representations (the $r_a$-representations) in terms of cylinders defined by
\[\Delta^{r_a}_{c_1c_2...c_m}=
\{x: x=\Delta^{r_a}_{c_1c_2...c_ma_1a_2...a_n...},  a_n\in A\},\]
We analyze their properties in detail, including the specific nature of overlaps.

We present results on the structural, variational, topological, metric and partially fractal properties of the function defined by
\[f\left(x=\sum_{n=1}^{\infty}\frac{\alpha_n}{(r+1)^n}\right)=
\Delta^{r_a}_{\alpha_1\alpha_2...\alpha_n...},\alpha_n \in A.\]
We prove the function is continuous at all points of the interval $[0,1]$ that have a unique representation in the classical numeral system on the base $r+1$ and prove  the function is discontinuous at points of a countable everywhere dense set in $[0,1]$. Furthermore, we show that the function is nowhere monotonic and has unlimited variation.

In the particular case $r=1$ and $a=\frac{1+\sqrt{5}}{2}$, we specify fractal level sets with Hausdorff--Besicovitch dimension not less than $-\log_a2$.
\end{abstract}

\maketitle
\section*{Introduction}

Many functions defined on an interval exhibit a locally complicated structure with respect to monotonicity, variation, differentiability, topological–metric, and fractal properties.

In addition to continuous nowhere monotone, nowhere differentiable and singular functions we classify as locally complicated those functions whose set of discontinuity points is countable and everywhere dense in the interval (the domain of definition). In the paper we study one class of such functions.

To define the function we use a  system of number representations with a redundant alphabet and in general, a non-integer base. 
Such systems have been studied in many papers~\cite{1,2,pr3,g4,der1,14_Komornik,1_dist_rv,pratmakarch,pratsmykyt,prats_dop_acad}, 
dealing with various problems, mainly in metric theory (including fractal aspects) 
and probabilistic number theory. Their application to problems of function theory is not observed in these works.

Let $1<a$ be a fixed real number and let $r \in N$ be such that $a-1 \leq r$. 
Let $A=\{0,1,\dots,r\}$ be an alphabet, and let $L = A \times A \times \cdots$ denote 
the set of sequences with elements from the alphabet $A$.

We consider the mapping $\gamma: L \to \left[0,\frac{r}{a-1}\right]$ defined by
\[
\gamma((\alpha_n))=\sum_{n=1}^{\infty}\frac{\alpha_n}{a^n}
\equiv \Delta^{r_a}_{\alpha_1\alpha_2\ldots\alpha_n\ldots}
= x \in \left[0,\frac{r}{a-1}\right], \quad (\alpha_n)\in L.
\]
We call the symbolic expression $\Delta^{r_a}_{\alpha_1\alpha_2\ldots\alpha_n\ldots}$ 
the \emph{$r_a$-representation} of the number $x$, and refer to $\alpha_n$ as 
the \emph{$n$th digit} of this representation.
We point out that the case $r=1$ is of particular interest, as the alphabet $A$ is minimal.

If $a = r + 1$, then the $r_a$-representation coincides with the classical representation of numbers in a positional numeral system with natural base $a$, which has zero redundancy (each number admits at most two representations, and only a countable set of numbers admits two).

\begin{definition}
For rational $a$, we call a number $x$ $r_a$-rational if it admits an $r_a$-representation with period $(0)$.
\end{definition}

It is clear that if $a$ is rational, then the $r_a$-representation is also rational.

We see that every $r_a$-rational number is rational. However, not every rational number is $r_a$-rational. For example, a number $x$ with a purely periodic $r_a$-representation
$\Delta^{r_a}_{(c_1\ldots c_p)}$, whose period contains at least two distinct digits, is rational but not $r_a$-rational. Indeed, we can represent the number $x$ as the value of the expression
$$
x=\Delta^{r_a}_{(c_1\ldots c_p)}=\left( \frac{c_1}{a}+\ldots + \frac{c_p}{a^p} \right)+\frac{1}{a^p}\left( \frac{c_1}{a}+\ldots + \frac{c_p}{a^p} \right)+\frac{1}{a^{2p}}\left( \frac{c_1}{a}+\ldots + \frac{c_p}{a^p} \right)+\ldots=\frac{c}{1-\frac{1}{a^p}},
$$
where $c=\frac{c_1}{a}+\ldots + \frac{c_p}{a^p}$ is a rational number, but (in general) it is not $r_a$-rational.

The condition $r \geq a$ guarantees the existence of numbers that admit more than two representations, and even a continuum of such numbers, which yields nonzero redundancy of the given $(r+1)$-symbol encoding of numbers. Researchers have studied such systems since 1957; the first contributions in this direction appeared in \cite{1,2}.

Today systems for encoding real numbers with nonzero redundancy receive intensive study, and hundreds of works address them \cite{pr3, g4, der1}. Most of these systems involve a base $a$ that is not an integer. Within this framework, researchers consider various objects and solve diverse problems. Several relatively comprehensive surveys summarize the results of these studies~\cite{14_Komornik}, particularly in number theory and fractal analysis of sets.

This number-theoretic topic closely relates to other directions of modern research, including the geometry of numerical series, the theory of infinite Bernoulli convolutions, dynamical systems, and the theory of functions with complicated local structure.

We view systems with nonzero redundancy as a powerful tool for constructing and studying mathematical objects with locally complex topological–metric structure, in particular functions. In this paper, we aim to demonstrate these possibilities.

\section{Geometry of the $r_a$-representation}

Each pair of parameters $a, r$ generates its own distinctive geometry and a specific pattern of overlaps. The properties of cylindrical sets partially reveal this structure. The following cases deserve particular attention: $a$ is a natural number; $a$ is rational; $a$ is irrational.

\begin{definition}
Let $(c_1,c_2,\ldots, c_k)$ be a fixed ordered tuple of elements of the alphabet.
We call the set
 $$\Delta^{r_a}_{c_1c_2...c_k}=\{x: x= \Delta^{r_a}_{c_1...c_k\alpha_1\alpha_2...\alpha_n...}, (\alpha_n)\in L\}.$$
a cylinder ($r_a$-cylinder) of rank $k$ with the base $c_1c_2...c_k$.
\end{definition}

We call the cylinders $\Delta^{r_a}_{c_1...c_{m-1}j}$ and $\Delta^{r_a}_{c_1...c_{m-1}[j+1]}$, $j\in {0,1,...,r-1}$, adjacent; they belong to the same cylinder of the previous rank $\Delta^{r_a}_{c_1...c_{m-1}}$.

Cylinders have the following properties:

1) $\Delta^{r_a}_{c_1...c_k}=\Delta^{r_a}_{c_1...c_k0}\cup\Delta^{r_a}_{c_1...c_k1}\cup...\cup\Delta^{r_a}_{c_1...c_kr}$;

2) $\Delta^{r_a}_{c_1...c_k}=[u;u+d]$, де $u=\sum\limits_{i=1}^{k}\frac{c_i}{a^i}$, $d=\frac{r}{a^{k}\cdot (a-1)}$;

3)  a cylinder length: $|\Delta^{r_a}_{c_1...c_k}|=d=\frac{r}{a^k\cdot (a-1)}\to 0$ ($k\to \infty$);

4)  $\bigcap\limits_{k=1}^{\infty}\Delta^{r_a}_{c_1...c_k}=\Delta^{r_a}_{c_1c_2...c_k...}=x$ for any sequence $(c_k)\in L$;

5) $\min \Delta_{c_1c_2\ldots c_kc}< \min \Delta_{c_1c_2\ldots c_k[c+1]},$ $0\leqslant c\leqslant r-1$;

6) $\Delta^{r_a}_{\alpha_1...\alpha_k}=\Delta^{r_a}_{\beta_1...\beta_k}$ $\Leftrightarrow$ $\sum\limits_{i=1}^{k}a^{-i}(\alpha_i-\beta_i)=0$;

7) $\Delta^{r_a}_{c_1...c_kc}\cap \Delta^{r_a}_{c_1...c_k[c+1]}=[\Delta^{r_a}_{c_1...c_k[c+1](0)};\Delta^{r_a}_{c_1...c_kc(r)}]\neq \varnothing;$

8) an overlap length:
$$|\Delta^{r_a}_{c_1...c_{k-1}c}\cap \Delta^{r_a}_{c_1...c_{k-1}[c+1]}|=
\Delta^{r_a}_{c_1...c_{k-1}c(r)}-
\Delta^{r_a}_{c_1...c_{k-1}[c+1](0)}=
\dfrac{r-a+1}{a^k(a-1)};$$

9) the condition $\Delta^{r_a}_{c_1\ldots c_{k-1}c}\cap \Delta^{r_a}_{c_1\ldots c_{k-1}[c+1]}=\Delta^{r_a}_{c_1\ldots c_{k-1}cr}=\Delta^{r_a}_{c_1\ldots c_{k-1}[c+1]0}$
 is equivalent to 
 $$\dfrac{c}{a^k}+\dfrac{r}{a^{k+1}}=\dfrac{c+1}{a^k}, \text{~~that is~~} r=a,$$
which may hold only for integer values of $a$;

10) an equality $|\Delta^{r_a}_{c_1\ldots c_kc_{k+1}}|=\frac{1}{2}|\Delta^{r_a}_{c_1\ldots c_{k}}|$ holds only under the condition $a=2$;

11) the condition $\Delta^{r_a}_{c_1\ldots c_{k-1}c}\cap \Delta^{r_a}_{c_1\ldots c_{k-1}[c+1]}=\Delta^{r_a}_{c_1\ldots c_{k-1}c\underbrace{r\ldots r}_m}=\Delta^{r_a}_{c_1\ldots c_{k-1}[c+1]\underbrace{0\ldots 0}_m}$
 is equivalent to 
 $$\dfrac{c}{a^k}+\dfrac{r}{a^{k+1}}+\ldots+\dfrac{r}{a^{k+m}}=\dfrac{c+1}{a^k}, \text{~~that is~~} r=\dfrac{a^m(a-1)}{a^m-1}.$$
\begin{remark}
Properties 8)--11) reflect the specific structure of overlaps among $r_a$-cylinders.
\end{remark}

\section{A Special Case}

Let us consider the case $r=2$. Then $1 < a < 3$.

\begin{theorem}
If $r=2$, then the set 
\[C\equiv C[r_a;\{0,2\}]=\{x:~x=\Delta^{r_a}_{\alpha_1\alpha_2...\alpha_n...},
  \alpha_n\in \{0,2\}\; n\in N\}\]
is a perfect, nowhere dense, self-similar Cantor-type set, whose Hausdorff–Besicovitch dimension equals $-\log_a 2$.
\end{theorem}

\begin{proof}
Using the $r_a$-cylinders, we can easily describe the structure of the set $C$, namely:
  
  1. $C\subset [\Delta^{r_a}_0\cup\Delta^{r_a}_{2}]$, moreover $\Delta^{r_a}_0\cap\Delta^{r_a}_{2}=\emptyset$, $$\max\{\Delta^{r_a}_{0}\}=\Delta^{r_a}_{0(2)}\in C, \;\; \min\{\Delta^{r_a}_{2}\}=\Delta^{r_a}_{2(0)}\in C.$$
  
  2. $\Delta^{r_a}_{c_1...c_m}\supset(\Delta^{r_a}_{c_1...c_m0}\cup
  \Delta^{r_a}_{c_1...c_m2})$, \;\; $\Delta^{r_a}_{c_1...c_m0}\cap\Delta^{r_a}_{c_1...c_m2}=\emptyset$.
  
  3. $C\subset C_n$,\;$C_n=\bigcup\limits_{\alpha_1\in V}...\bigcup\limits_{\alpha_n\in V}\Delta^{r_a}_{\alpha_1...\alpha_n}$, $V=\{0,2\}$ $\forall n\in N$.
  
  4. $C=\bigcap\limits_{n=1}^{\infty}C_n$.
  
According to the well-known theorem on the structure of perfect sets, we conclude that the set $C$ is perfect and nowhere dense (it contains no however small interval).

The set $C$ is self-similar since
 \[C=[\Delta^{r_a}_{0}\cap C]\cup[\Delta^{r_a}_{1}\cap C],\]
  where $C\stackrel{k}{\sim}(\Delta^{r_a}_{0}\cap C)\cong
  (\Delta^{r_a}_{2}\cap C)$, $k=\frac{r}{a-1}:|\Delta^{r_a}_{0}|=\frac{1}{(a-1)}:\frac{1}{a(a-1)}=a.$
Its self-similarity dimension is the solution of the equation $2 \cdot a^x = 1$, that is, $x = -\log_a 2$. This value coincides with the Hausdorff–Besicovitch dimension, since the set $C$ satisfies the open set condition. 
\end{proof}

\section{A function with fractal level sets properties}

We consider the function $f$ defined by $f(x=\Delta^{r+1}_{\alpha_1\alpha_2\ldots\alpha_n\ldots})=\Delta^{r_a}_{\alpha_1\alpha_2\ldots\alpha_n\ldots}$, 
where
$$\Delta^{r+1}_{\alpha_1\alpha_2\ldots\alpha_n\ldots}=
\sum^{\infty}_{n=1}\dfrac{\alpha_n}{(r+1)^n},\; (\alpha_n)\in L.$$

Since a countable set of numbers in the $(r+1)$-adic system has two formally distinct representations:
$\Delta^{r+1}_{\alpha_1\alpha_2\ldots\alpha_n(0)}=\Delta^{r+1}_{\alpha_1\alpha_2\ldots[\alpha_n-1](r)}$ 
(we call these numbers$(r+1)$-binary), we agree for the function $f$ to be well-defined to use only the first of these $(r+1)$-binary representations.

\begin{theorem}
The function $f$ is continuous at every $(r+1)$-unary point, and it is continuous at a $(r+1)$-binary point if and only if $a = r+1$.
\end{theorem}

\begin{proof}
If the function $f$ is continuous on an interval, then it is continuous at every point of that interval. Let $x_0 = \Delta^{r+1}_{\alpha_1\alpha_2...\alpha_n...}$ be a $(r+1)$-unary point, and let $f(x_0)=\Delta^{r_a}_{\alpha_1\alpha_2...\alpha_n...}$. 
Consider a point $x \neq x_0$. Then, since $x \neq x_0$, there exists an index $k$ such that $\alpha_k(x) \neq \alpha_k(x_0)$, while $\alpha_i(x) = \alpha_i(x_0)$ for all $i < k$. Moreover, the condition $k \to \infty$ is equivalent to $x \to x_0$.
To prove the continuity of the function $f$ at the point $x_0$, we show that  
\[\lim\limits_{x\to x_0}|f(x)-f(x_0)|=0.\]
By the definition of the function $f$, we have
 \begin{align*}
    \lim\limits_{x\to x_0}|f(x)-f(x_0)|=&
    \lim\limits_{x\to x_0}|f(\Delta^{r_a}_{\alpha_1\alpha_2...\alpha_{k-1}\alpha'_k\alpha'_{k+1}...})-
    f(\Delta^{r_a}_{\alpha_1\alpha_2...\alpha_{k-1}\alpha_k\alpha_{k+1}...})|=\\
    =&\lim\limits_{k\to\infty}|\Delta^{r_a}_{\alpha_1\alpha_2...\alpha_{k-1}\alpha'_k\alpha'_{k+1}...}-
    \Delta^{r_a}_{\alpha_1\alpha_2...\alpha_{k-1}\alpha_k\alpha_{k+1}...}|\leq\\
    \leq&\lim\limits_{k\to\infty}\frac{1}{a^{k-1}}|\Delta^{r_a}_{\alpha'_k\alpha'_{k+1}...}-
    \Delta^{r_a}_{\alpha_k\alpha_{k+1}...}|=0.
  \end{align*}
Therefore, the function $f$ is continuous at a $(r+1)$-unary point.

Continuity of the function at every $(r+1)$-binary point is equivalent to satisfying the equality for every $k \in N$:
 \[f(\Delta^{r+1}_{\alpha_1\alpha_2...\alpha_{k-1}\alpha_k(0)})=
  f(\Delta^{r+1}_{\alpha_1\alpha_2...\alpha_{k-1}[\alpha_k-1](r)}), \mbox{ for any } k\in N.\]
We can rewrite the last equality in the form:
 \[\sum\limits_{i=1}^{k-1}\frac{\alpha_i}{a^i}+\frac{\alpha_k}{a^k}=
  \sum\limits_{i=1}^{k-1}\frac{\alpha_i}{a^i}+\frac{\alpha_k-1}{a^k}+\frac{r}{a^2-a},\]
wherefrom
\[\frac{\alpha_k}{a^k}=
  \frac{\alpha_k-1}{a^k}+\frac{r}{a^2-a},\]
which is possible only if $a = r+1$. That is, the function $f$ is continuous at every $(r+1)$-binary point if and only if $a = r+1$.
\end{proof}

\begin{remark}
Note that when $a = r+1$, that is, when the $r_a$-representation coincides with the $(r+1)$-representation, we have $f(x) = x$.
\end{remark}

\begin{theorem}
If $a < r+1$, the function $f$ is nowhere monotone.
\end{theorem}  

\begin{proof}
The possible cases are: 1) $1 < a \leq r$; 2) $a \in (r, r+1)$. To prove that the function is nowhere monotone, it suffices to show its non-monotonicity on an arbitrary cylinder of rank $m$.
To be specific, we consider the cylinder
$\Delta^{r+1}_{c_1...c_m}=[\Delta^{r+1}_{c_1...c_m(0)};\Delta^{r+1}_{c_1...c_m (r)}]$.

1) Let $1 < a \leq r$, so that $r+1 - a \geq 1$. Consider, on the cylinder
$\Delta^{r+1}_{c_1c_2...c_m}$ points  
  $x_1=\Delta^{r+1}_{c_1...c_m(0)}$,
  $x_2=\Delta^{r+1}_{c_1...c_mr0r(r-1)}$,
  $x_3=\Delta^{r+1}_{c_1...c_mr1(0)}$. 
We have $x_1 < x_2 < x_3$. Let us consider the differences
\begin{align*}
  f(x_2)-f(x_1)=&f(x_2)>0,\\
  f(x_2)-f(x_3)=&
  \frac{r}{a^{m+1}}+\frac{r}{a^{m+3}}+\frac{r-1}{a^{m+3}(a-1)}-
  \frac{r}{a^{m+1}}-\frac{1}{a^{m+2}}=\\
  =&\frac{r(a-1)+r-1-a(a-1)}{a^{m+3}(a-1)}=
  \frac{ra-r+r-1-a^2+a}{a^{m+3}(a-1)}=\\
  =&\frac{a(r+1-a)-1}{a^{m+3}(a-1)}\geq\frac{a-1}{a^{m+3}(a-1)}>0.
\end{align*}
Since $(f(x_2) - f(x_1))(f(x_2) - f(x_3)) > 0$, the function $f$ is non-monotone on the cylinder $\Delta^{r+1}_{c_1c_2...c_m}$, and due to the arbitrary choice of $\Delta^{r+1}_{c_1c_2...c_m}$, it is nowhere monotone on the entire domain.

2) Let $a \in (r, r+1)$, so that $r+1 - a > 0$.

$\Delta^{r+1}_{c_1...c_m}=[\Delta^{r+1}_{c_1...c_m(0)};\Delta^{r+1}_{c_1...c_m (r)}]$. 
Consider the points
$x_1=\Delta^{r+1}_{c_1...c_m(0)}$, $x_2=\Delta^{r+1}_{c_1...c_m0\underbrace{r...r}_n (0)}$, $x_3=\Delta^{r+1}_{c_1...c_m1(0)}$.
Clearly, we have $x_1 < x_2 < x_3$. Let us examine the differences
  $$
  f_2(x)-f_1(x)=\frac{1}{a^m}
  \left(\frac{r}{a}+...+\frac{r}{a^n}\right)>0,
  $$
  $$
  f(x_3)-f(x_2)=\frac{1}{a^m}
  \left(\frac{1}{a}-\frac{r}{a^2}\cdot\frac{1}{1-\frac{1}{a}}
  \left(1-\frac{1}{a^n}\right)\right)=
  $$
  $$
  =\frac{1}{a^m}
  \left(\frac{1}{a}-\frac{r}{a(a-1)}
  (1-\frac{1}{a^n})\right)=\frac{1}{a^m}
  \left(\frac{a-1-r}{a(a-1)}+\frac{r}{a^{n+1}}(a-1)\right).$$
Let $\frac{r-(a-1)}{a(a-1)}=c>0$, then we have
  \[a^m(f(x_3)-f(x_2))=\frac{a-1-r}{a(a-1)}+\frac{r}{a^{n+1}(a-1)}=
  \frac{r}{a^{n+1}(a-1)}-c.\]
Since $\frac{r}{a^{n+1}(a-1)}$  monotonically tends to 0, there exists 
$n_0$ such that for all $n>n_0$ the inequality $\frac{r}{a^{n+1}(a-1)}<c$ holds. For such 
$n$, we have  $f(x_3)-f(x_2)<0$. Since
\[(f(x_2)-f(x_1))(f(x_3)-f(x-2))<0,\] 
the function $f(x)$ is non-monotone on the cylinder $\Delta^{r_a}_{c_1...c_m}$. Hence, 
$f$ is nowhere monotone.
\end{proof}

\begin{theorem}
If $a < r+1$, then the function $f$ has unbounded variation.
\end{theorem}

\begin{proof}
Let $\Delta \equiv [0;\frac{r}{a-1}]$.
The oscillation of the function $f$ on the cylinder
$\Delta^{r+1}_{\alpha_1\alpha_2...\alpha_n}$ equals the length of the cylinder $\Delta^{r_a}_{\alpha_1\alpha_2...\alpha_n}$, which is the image of the cylinder $\Delta^{r+1}_{\alpha_1\alpha_2...\alpha_n}$. Therefore, the variation of the function $f$ exceeds the total length of the image cylinders, that is,
\[V(f)>V_k=\sum\limits_{\alpha_1=0}^{r}
\sum\limits_{\alpha_2=0}^{r}\ldots \sum\limits_{\alpha_n=0}^{r}
|\Delta^{r_a}_{\alpha_1\alpha_2\ldots\alpha_n}|.\]
Since
\[V_1\equiv\sum\limits_{\alpha_1=0}^{r}
|\Delta^{r_a}_{\alpha_1}|=
(r+1)\frac{r}{a(a-1)}>\frac{ar}{a(a-1)}=\frac{r}{a-1}=|\Delta|> 1.\]
Then 
\[k\equiv\frac{V_1}{|\Delta|}>1 \mbox{ і } V_1=k|\Delta|.\]
Similarly,
\[\sum\limits_{\alpha_2=0}^{r}|\Delta^{r_a}_{c_1\alpha_2}|
=|\Delta^{r_a}_{c_1}|\sum\limits_{\alpha_2=0}^{r}
|\Delta^{r_a}_{\alpha_2}|>|\Delta^{r_a}_{c_1}|.\]
Thus
\[V_2\equiv\sum\limits_{\alpha_1=0}^{r}
\sum\limits_{\alpha_2=0}^{r}
|\Delta^{r_a}_{\alpha_1\alpha_2}|=
\sum\limits_{\alpha_1=0}^{r}|\Delta^{r_a}_{\alpha_1}|
\sum\limits_{\alpha_2=0}^{r}|\Delta^{r_a}_{\alpha_2}|
>k^2|\Delta|,\]
\[\ldots\;\;\ldots\;\;\ldots\;\;\ldots\;\;\ldots\;\;\ldots\]
\[V_n=\sum\limits_{\alpha_1=0}^{r}
\sum\limits_{\alpha_2=0}^{r}\ldots \sum\limits_{\alpha_n=0}^{r}
|\Delta^{r_a}_{\alpha_1\alpha_2\ldots\alpha_n}|>
k^n|\Delta|,\]
then
\[V(f)\geq \lim\limits_{n\to\infty} V_n>\lim\limits_{n\to\infty}k^n|\Delta|=\infty.\]
Therefore, the function $f$ has unbounded variation.
\end{proof}

\section{Two-symbol systems}

The case $r=1$ merits special attention due to the minimality of the alphabet. We focus on this setting and assume $1 < a < 2$. In this case, we have
\[|\Delta^{r_a}_{c_1c_2\ldots c_m}|=\dfrac{1}{a^m(a-1)};\]
\[
|\Delta^{r_a}_0 \cap \Delta^{r_a}_1|=|[\Delta^{r_a}_{1(0)};\Delta^{r_a}_{0(1)}]|=
\Delta^{r_a}_{0(1)} - \Delta^{r_a}_{1(0)}= \frac{2-a}{a(a-1)};
\]
\begin{align*}
|\Delta^{r_a}_{c_1\ldots c_{m-1}0} \cap \Delta^{r_a}_{c_1\ldots c_{m-1}1}|&=|[\Delta^{r_a}_{c_1\ldots c_{m-1}1(0)};\Delta^{r_a}_{c_1\ldots c_{m-1}0(1)}]|=\Delta^{r_a}_{c_1\ldots c_{m-1}0(1)} - \Delta^{r_a}_{c_1\ldots c_{m-1}1(0)}=\\
&=\frac{2-a}{a^m(a-1)}.
\end{align*}

We note that certain particular cases are of independent interest. For example, the equality 
$$|\Delta^{r_a}_{c_1\ldots c_{m-1}0} \cap \Delta^{r_a}_{c_1\ldots c_{m-1}1}|=\frac{1}{2}|\Delta^{r_a}_{c_1\ldots c_{m}}|$$ holds only when $a=\frac{3}{2}$. 

In this case we have the equality $\frac{2-a}{a^m(a-1)}=\frac{1}{2}\cdot \frac{1}{a^{m}(a-1)}$, which is equivalent to the condition $a=\frac{3}{2}$.
Under this condition, $[0;\frac{r}{a-1}]=[0;2]$. This case calls for separate consideration.

Another \textbf{unique case is $r=1$, $a=\frac{1+\sqrt{5}}{2}$}. Let us determine under which conditions two adjacent cylinders of rank $m$ intersect in a cylinder of rank $(m+2)$, that is, when
\begin{equation}\label{mitka}
\Delta^{r_a}_{c_1...c_{m-1}0}\cap\Delta^{r_a}_{c_1...c_{m-1}1}=
\Delta^{r_a}_{c_1...c_{m-1}011}=
\Delta^{r_a}_{c_1...c_{m-1}100}.
\end{equation}
The cylinders
$\Delta^{r_a}_{c_1...c_{m-1}011}$ and $\Delta^{r_a}_{c_1...c_{m-1}100}$
have the same rank, so for the cylinders to coincide it suffices that their initial segments match, namely
 \[\min\Delta^{r_a}_{c_1...c_{m-1}011}-
 \min\Delta^{r_a}_{c_1...c_{m-1}100}=0.\]
And this is equivalent to the equality
\[\frac{1}{a^{m-1}}\left(\frac{1}{a^2}+\frac{1}{a^3}-\frac{1}{a}\right)=0,\]
that is $a+1-a^2=0$, which holds only for the single positive value $a=\frac{1+\sqrt{5}}{2}$.

 Since 
 $$\max\Delta^{r_a}_{c_1...c_{m-1}0}=
 \max\Delta^{r_a}_{011}, \quad
 \min\Delta^{r_a}_{c_1...c_{m-1}1}=
 \min\Delta^{r_a}_{100},
 $$
 then the necessary and sufficient condition for satisfying~\eqref{mitka} is $a=\frac{1+\sqrt{5}}{2}$.
 \begin{remark}
 The ordered triples of digits $(1,0,0)$ and $(0,1,1)$ in the $r_a$-representation of a number, for $r=1$ and $a=\frac{1+\sqrt{5}}{2}$, are interchangeable when they appear as consecutive digits in the representation.
 \end{remark}
 \begin{theorem}
  If $r=1$ and $a = \frac{1+\sqrt{5}}{2}$, then the level set $f^{-1}(y_0)$ for $y_0 = \Delta^{r_a}_{(100)}$ is a continuum and has a fractal dimension of at least $\frac{1}{3}$.
 \end{theorem}

\begin{proof}
According to the previous remark, the number $\Delta^{r_a}_{(100)}$, whose representation has period $(100)$, admits a continuum of distinct $r_a$-representations, since each consecutive triple $(1,0,0)$ can be replaced by the alternative $(0,1,1)$ without changing the value. Therefore, the preimage set $y_0=\Delta^{r_a}_{(100)}$ contains a subset of numbers from the interval $[0,1]$ whose classical 8-adic representations use only the digits $4$ and $5$.
  
In fact,  
  \begin{align*}
    x=&\left(\frac{\alpha_1}{2}+\frac{\alpha_2}{2^2}+\frac{\alpha_3}{2^3}\right)+
\left(\frac{\alpha_4}{2^4}+\frac{\alpha_5}{2^5}+\frac{\alpha_6}{2^6}\right)+
\left(\frac{\alpha_7}{2^7}+\frac{\alpha_8}{2^8}+\frac{\alpha_9}{2^9}\right)+...=\\ =&\frac{4\alpha_1+2\alpha_2+\alpha_3}{2^3}+
\frac{4\alpha_4+2\alpha_5+\alpha_6}{2^6}+
\frac{4\alpha_7+2\alpha_8+\alpha_9}{2^9}+...
  \end{align*}
Since 
  \[\frac{1}{2}+\frac{0}{2^2}+\frac{0}{2^3}=\frac{4}{2^3},
\frac{0}{2}+\frac{1}{2^2}+\frac{1}{2^3}=\frac{5}{2^3},\]
then one of the preimages of the number $y_0=\Delta^{r_a}_{(100)}$ is
$x_0=\Delta^8_{(4)}$, another is $\Delta^8_{(5)}$ and $\Delta^{8}_{(45)}$ and so on. Thus, the entire Cantor-type set
\[C[8;\{4,5\}]=\{x:~ x=\Delta^8_{a_1a_2...a_n...}, \mbox{ where } a_n\in \{4,5\}\}\]
belongs to the level set. Its self-similarity dimension coincides with its Hausdorff–Besicovitch dimension and equals $\log_82=\frac{1}{3}$.
 \end{proof}


\begin{thebibliography}{99}
\bibitem{1}{\it A. Renyi} Representations for real numbers and their ergodic properties, Acta Math. Acad. Sci. Hung., 8 (1957), 477-493.

\bibitem{2}{\it W. Parry} On the $\beta$-expansions of real numbers, Acta Math. Acad. Sci. Hung., 11 (1960),401-416.

\bibitem{pr3}{\it Pratsiovytyi M.V.} Two symbol encodign of real numbres and its aplications. Nauk. Dumka, Kyiv 2022,  316 p.
    
\bibitem{g4}
{\it Goncharenko Ya.V., Mykytiuk I.O} Representations of real numbers in numeral systems with redundant set of digits and their applications. Nauk. Chasop. Nats. Pedagog. Univ. Mykhaila Dragomanova. Ser 1. Fiz.-Mat. Nauky, 2004, \textbf{5}, 255--275. (in Ukrainian).

\bibitem{der1}{\it Derong Kong, Wenxia Li, Fan Lu, Zhiqiang Wang, Jiayi Xu} Univoque bases of real numbers: Local dimension, Devil's staircase and isolated points,
Advances in Applied Mathematics, Volume 121, 2020,

\bibitem{14_Komornik}
{\it Zou Y., Lu J., Komornik V.} {Hausdorff dimension of multiple expansions}. Journal of Number
Theory, 2022, 233, pp.198-227. ff10.1016/j.jnt.2021.06.009

\bibitem{1_dist_rv} {\it Pratsiovytyi M.V., Ratushniak S.P.} {Singular distributions of random variables with independent digits of representation in numeral system with natural base and redundant alphabet}. Matematychni Studii 2025, \textbf{63} (2), 199--209. doi:http://doi.org/10.30970/ms.63.2.199-209
    
\bibitem{pratmakarch} {\it Pratsiovytyi M.V., Makarchuk O.P.} Distribution of random variable represented by a binary  fraction with three identically distributed redundant digits, Ukrainian Mathematical Journal
2022, 66 (1),  79--88

\bibitem{pratsmykyt} {\it Mykytyuk I.O, Pratsiovytyi M.V.} The binary numeral system with two redundant digits and its corresponding metric theory of numbers, Scientific Notes of M.P. Dragomanov National Pedagogical
University. Series: Physical and Mathematical Sciences,
	        2003, vol. 4, 270--290.

\bibitem{prats_dop_acad} {\it Pratsiovytyi M.V.}
Convolutions of singular distributions, Reports of the National Academy of Sciences of Ukraine, 1996,
	vol. 5 (5), 36--42.
\end{thebibliography}
\end{document}